\documentclass[12pt]{article}
\usepackage{amsmath,amsthm,amscd,amsfonts}
\newcommand{\sect}[1]{\section{#1}\setcounter{equation}{0}}

\font\mbn=msbm10 scaled \magstep1
\font\mbs=msbm7 scaled \magstep1
\font\mbss=msbm5 scaled \magstep1
\newfam\mbff
\textfont\mbff=\mbn
\scriptfont\mbff=\mbs
\scriptscriptfont\mbff=\mbss\def\mbf{\fam\mbff}

\def\Co{{\mbf C}}

\def\Di{{\mbf D}}

\newtheorem{Th}{Theorem}[section]

\newtheorem{R}[Th]{Remark}

\author{Alexander Brudnyi\thanks{Research supported in part by NSERC.\newline 
2000 {\em Mathematics Subject Classification}. Primary 37L10,
Secondary 34C07.
\newline 
{\em Key words and phrases}. Center problem, the group of rectangular paths, iterated integrals.
}\\
Department of Mathematics and Statistics\\
University of Calgary, Calgary\\
Canada}

\title{Center Problem for the Group of Rectangular Paths}

\date{} 

\begin{document} 


\maketitle

\begin{abstract}
{We solve the center problem for ODEs $\frac{dv}{dx}=\sum_{i=1}^{\infty}a_{i}(x)\,v^{i+1}$ such that the first integrals of vectors of their coefficients determine rectangular paths in finite dimensional complex vector spaces.}
\end{abstract}

\sect{\hspace*{-1em}. Introduction}
{\bf 1.1.} The classical Poincar\'{e} Center-Focus problem for planar vector fields 
\begin{equation}\label{e3}
\frac{dx}{dt}=-y+F(x,y),\ \ \ \frac{dy}{dt}=x+G(x,y),
\end{equation}
where $F$ and $G$ are real polynomials of a given degree without constant and linear terms asks about conditions on $F$ and $G$ under which all trajectories of (\ref{e3}) situated in a small neighbourhood of $0\in\mathbb R^{2}$ are closed. It can be reduced  passing to polar coordinates $(x,y)=(r\cos\phi, r\sin\phi)$ in (\ref{e3}) and expanding the right-hand side of the resulting equation as a series in $r$ (for $F$, $G$ with sufficiently small coefficients) to the center problem for the ordinary differential equation
\begin{equation}\label{e1}
\frac{dv}{dx}=\sum_{i=1}^{\infty}a_{i}(x)v^{i+1},\ \ \ x\in [0,2\pi],
\end{equation}
whose coefficients are trigonometric polynomials depending polynomially on the coefficients of $F$ and $G$.

More generally, consider equation (\ref{e1}) with coefficients $a_{i}$ from the Banach space $L^{\infty}(I_{T})$ of bounded measurable complex-valued functions on $I_{T}:=[0,T]$ equipped with the supremum norm. Condition
\begin{equation}\label{e2}
\sup_{x\in I_{T},\, i\in\mathbb N}\sqrt[i]{|a_{i}(x)|}<\infty
\end{equation}
guarantees that (\ref{e1}) has Lipschitz solutions on $I_{T}$ for all sufficiently small initial values.
By $X$ we denote the complex Fr\'{e}chet space of sequences $a=(a_{1},a_{2},\dots)$ satisfying (\ref{e2}).
We say that equation (\ref{e1}) determines a {\em center} if every solution $v$ of (\ref{e1}) with a sufficiently small initial value satisfies $v(T)=v(0)$. By ${\cal C}\subset X$ we denote the set of centers of (\ref{e1}). The center problem is: {\em given $a\in X$ to determine whether it belongs to ${\cal C}$}. In particular, for system (\ref{e3}) with $F$ and $G$ being real polynomials of degree $d$ the coefficients of the corresponding equation (\ref{e1}) determine a polynomial map $P_d:\mathbb R^{N(d)}\to X$ where $N(d):=(d-1)(d+4)$ is the dimension of the space of coefficients of pairs $(F,G)$. Then the Poincar\'{e} problem is {\em to characterize the elements of the set} $P_d^{-1}(P_d(\mathbb R^{N(d)})\cap {\cal C})\subset\mathbb R^{N(d)}$.\\
\\
{\bf 1.2.} Let us rehearse some results on the structure of the set ${\cal C}$ obtained in [Br1]-[Br2].
For each $a=(a_1,a_2,\dots)\in X$ the first integral $\widetilde a(x):=(\int_0^x a_1(s)\,ds,\int_0^x a_2(s)\,ds,\dots)$ determines a path $\widetilde a:[0,T]\to\Co^\infty$. We consider the set of paths with the standard operations of multiplication and taking the inverse.\footnote{The product of paths $\widetilde a\circ\widetilde b$ is the path obtained by translating $\widetilde a$ so that its beginning meets the end of $\widetilde b$ and then
forming the composite path. Similarly, $\widetilde a^{-1}$ is the path obtained by translating $\widetilde a$ so
that its end meets $0$ and then taking it with the opposite orientation.} Then we introduce similar operations $*$  and $^{-1}$ on $X$ so that the correspondence $a\mapsto\widetilde a$ is a monomorphism of semigroups.

Let $\Co\!\left\langle X_{1},X_{2},\dots\right\rangle$ be the associative algebra with unit $I$ of complex noncommutative polynomials in $I$ and free noncommutative variables $X_{1}, X_{2},\dots$ (i.e., there are no nontrivial relations between these variables). By $\Co\!\left\langle X_{1},X_{2},\dots\right\rangle\![[t]]$ we denote the associative algebra of formal power series in $t$ with coefficients from
$\Co\!\left\langle X_{1},X_{2},\dots\right\rangle$. Also, by ${\cal A}\subset \Co\!\left\langle X_{1},X_{2},\dots\right\rangle\![[t]]$ we denote the subalgebra of series $f$ of the form
\begin{equation}\label{e2.6}
f=c_{0}I+\sum_{n=1}^{\infty}\left(\ \!\sum_{i_{1}+\dots +i_{k}=n}c_{i_{1},\dots, i_{k}}X_{i_{1}}\cdots X_{i_{k}}
\right)t^{n}
\end{equation}
with $c_{0},c_{i_{1},\dots, i_{k}}\in\Co$ for all $i_{1},\dots, i_{k},\, k\in\mathbb N$.

By $G\subset{\cal A}$ we denote the closed subset of elements $f$ of form (\ref{e2.6}) with $c_{0}=1$.  We equip ${\cal A}$
with the adic topology determined by powers of the ideal ${\cal I}\subset {\cal A}$ of elements of form (\ref{e2.6}) with $c_0=0$. Then $(G,\cdot)$ is a topological group. Its Lie algebra ${\cal L}_{G}\subset {\cal A}$ is the vector space of elements of form (\ref{e2.6}) with $c_{0}=0$;  here for $f,g\in {\cal L}_{G}$ their product is defined by the formula $[f,g]:=f\cdot g-g\cdot f$. Also, the map $\exp:{\cal L}_{G}\to G$,
$\exp(f):=e^{f}=\sum_{n=0}^{\infty}\frac{f^{n}}{n!}$, is a homeomorphism.

For an element
$a=(a_{1},a_{2},\dots)\in X$ consider the equation 
\begin{equation}\label{e2.8}
F'(x)=\left(\ \!\sum_{i=1}^{\infty}a_{i}(x)\ \!X_i\, t^{i}\right)F(x),
\ \ \ \ \ x\in I_{T}.
\end{equation}
This can be solved by Picard iteration to obtain a solution $F_{a}: I_{T}\to G$, $F_{a}(0)=I$, whose coefficients in expansion in $X_{1}, X_{2},\dots$ and $t$ are Lipschitz functions on $I_{T}$.  We set 
\begin{equation}\label{e2.9}
E(a):=F_{a}(T),\ \ \ a\in X.
\end{equation}
Then
\begin{equation}\label{e2.10}
E(a*b)=E(a)\cdot E(b),\ \ \ a,b\in X.
\end{equation}
An explicit calculation leads to the formula
\begin{equation}\label{e2.11}
E(a)=I+\sum_{n=1}^{\infty}\left(\ \!\sum_{i_{1}+\cdots +i_{k}=n}I_{i_{1},\dots, i_{k}}(a)X_{i_{1}}\cdots X_{i_{k}}\right)t^{n}
\end{equation}
where
\begin{equation}\label{E2.1}
I_{i_{1},\dots, i_{k}}(a):=\int\cdots\int_{0\leq s_{1}\leq\cdots\leq s_{k}\leq T}a_{i_{k}}(s_{k})\cdots a_{i_{1}}(s_{1})\ \!ds_{k}\cdots ds_{1}
\end{equation}
are basic iterated integrals on $X$.

The kernel of the homomorphism $E:X\to G$ is called {\em the set of universal centers} of equation (\ref{e1}) and is denoted by ${\cal U}$. The elements of ${\cal U}$ are of a topological nature, see [Br1] for their description. The set of equivalence classes $G(X):=X/\sim$ with respect to the equivalence relation $a\sim b \leftrightarrow a*b^{-1}\in {\cal U}$ has the structure of a group so that the factor-map $\pi:X\to G(X)$ is an epimorphism of semigroups. Moreover, for each function $I_{i_{1},\dots, i_{k}}$ on $X$ there exists a function $\widehat I_{i_{1},\dots, i_{k}}$ on $G(X)$ such that $\widehat I_{i_{1},\dots, i_{k}}\circ\pi=I_{i_{1},\dots, i_{k}}$.
In particular, there exists a monomorphism of groups $\widehat E:G(X)\to G$ defined by $E=\widehat E\circ\pi$, i.e.,
\begin{equation}\label{e2.12}
\widehat E(g)=I+\sum_{n=1}^{\infty}\left(\ \!\sum_{i_{1}+\cdots +i_{k}=n}\widehat I_{i_{1},\dots, i_{k}}(g)X_{i_{1}}\cdots X_{i_{k}}\right)t^{n},\ \ \ g\in G(X). 
\end{equation}

We equip $G(X)$ with the weakest topology in which all functions
$\widehat I_{i_{1},\dots, i_{k}}$ are continuous. Then $G(X)$ is a topological group and $\widehat E$ is a continuous embedding. The completion of the image $\widehat E(G(X))\subset {\cal A}$ is called {\em the group of formal paths in} $\mathbb C^\infty$ and is denoted by $G_f(X)$. The group $G_f(X)$ is defined by the Ree shuffle relations for the iterated integrals. Its Lie algebra ${\cal L}_{Lie}$ consists of all {\em Lie elements} of ${\cal A}$, see [Br2] for details. 

Let $G[[r]]$ be the set of formal complex power series $f(r)=r+\sum_{i=1}^{\infty}d_{i}r^{i+1}$.
Let $d_{i}:G[[r]]\to\mathbb C$ be such that $d_{i}(f)$ is the $(i+1)$st coefficient in the series expansion of $f$. We equip $G[[r]]$ with the weakest topology in which all $d_{i}$ are continuous functions and consider the multiplication $\circ$ on $G[[r]]$ defined by the composition of series. Then $G[[r]]$ is a separable topological group. By $G_{c}[[r]]\subset G[[r]]$ we denote the subgroup of power series locally convergent near $0$ equipped with the induced topology. Next, we define the map $P:X\to G[[r]]$ by the formula
\begin{equation}\label{e3.6}
P(a)(r):=r+\sum_{i=1}^{\infty}\left(\,\sum_{i_{1}+\cdots +i_{k}=i}p_{i_{1},\dots, i_{k}}(i)\cdot I_{i_{1},\dots, i_{k}}(a)\right)r^{i+1}
\end{equation}
where
$$
p_{i_{1},\dots, i_{k}}(t):=(t-i_{1}+1)(t-i_{1}-i_{2}+1)\cdots (t-i+1).
$$
Then $P(a*b)=P(a)\circ P(b)$ and $P(X)=G_{c}[[r]]$. 
Moreover, let $v(x;r;a)$, $x\in I_{T}$, be the Lipschitz solution of equation (\ref{e1}) with initial value $v(0;r;a)=r$. Clearly for every $x\in I_{T}$ we have $v(x;r;a)\in G_{c}[[r]]$. It is shown in [Br1] that $P(a)=v(T;\cdot;a)$ (i.e., $P(a)$ is the {\em first return map} of (\ref{e1})). In particular, we have
\begin{equation}\label{e3.7}
a\in {\cal C}\ \Longleftrightarrow\ \sum_{i_{1}+\cdots +i_{k}=i}p_{i_{1},\dots, i_{k}}(i)\cdot I_{i_{1},\dots, i_{k}}(a)\equiv 0\ \ \ {\rm for\ all}\ \ \ i\in\mathbb N.
\end{equation}

Equation (\ref{e3.6}) implies that there exists a continuous homomorphism of groups $\widehat P: G(X)\to G[[r]]$ such that $P=\widehat P\circ\pi$. Identifying $G(X)$ with its image under $\widehat E$ we extend $\widehat P$ by continuity to $G_{f}(X)$ retaining the same symbol for the extension.

We set $\widehat{\cal C}:=\pi({\cal C})$ and define $\widehat{\cal C}_f$ as the completion of $\widehat E(\widehat{\cal C})$. Then $\widehat{\cal C}_f$ coincides with the kernel of the homomorphism $\widehat P$. The groups $\widehat{\cal C}$ and $\widehat{\cal C}_f$ are called {\em the groups of centers and formal centers} of equation (\ref{e1}).

It was established in [Br1], [Br2] that
\begin{equation}
\begin{array}{c}
\displaystyle
g\in \widehat{\cal C}_{f}\ \ \ \Longleftrightarrow\ \sum_{i_{1}+\cdots +i_{k}=i}p_{i_{1},\dots, i_{k}}\widehat I_{i_{1},\dots, i_{k}}(g)\equiv 0\ \ \ {\rm for\ all}\ \ \ i\in\mathbb N\ \ \ \Longleftrightarrow\\
\\
\displaystyle
\sum_{i_{1}+\cdots +i_{k}=i}p_{i_{1},\dots, i_{k}}(i)\cdot \widehat I_{i_{1},\dots, i_{k}}(g)=0\ \ \ {\rm for\ all}\ \ \ i\in\mathbb N
\end{array}
\end{equation}
and the Lie algebra ${\cal L}_{\widehat{\cal C}_{f}}\subset {\cal L}_{Lie}$ of $\widehat{\cal C}_{f}$
consists of elements
$$
\sum_{n=1}^{\infty}\left(\ \!\sum_{i_{1}+\dots +i_{k}=n}c_{i_{1},\dots, i_{k}}[X_{i_{1}},[X_{i_{2}},[\ \cdots , [X_{i_{k-1}},X_{i_{k}}]\cdots \ ]]]\right) t^{n}
$$
such that
\begin{equation}\label{e3.21}
\begin{array}{c}
\displaystyle 
\sum_{i_{1}+\cdots + i_{k}=n}c_{i_{1},\dots, i_{k}}\cdot\gamma_{i_{1},\dots, i_{k}}=0\ \ \ {\rm for\ all}\ \ \ n\in\mathbb N \ \ \ {\rm where}\ \ \ \gamma_{n}=1\ \ \ {\rm and}\\
\\
\gamma_{i_{1},\dots, i_{k}}=
(i_{k-1}-i_{k})(i_{k-1}+i_{k}-i_{k-2})\cdots
(i_{2}+\cdots +i_{k}-i_{1})\ \ \ {\rm for}\ \ \ k\geq 2.
\end{array}
\end{equation}
In particular, the map $\exp: {\cal L}_{\widehat{\cal C}_{f}}\to  \widehat{\cal C}_{f}$ is a homeomorphism.
\sect{\hspace*{-1em}. Main results}
{\bf 2.1.} Consider elements $g\in G_{f}(X)$ of the form 
\begin{equation}\label{eq2.1}
g=e^{h}\ \ \ {\rm where}\ \ \ h=\sum_{i=1}^{\infty}c_{i}X_{i}t^{i},\ \ \
c_{i}\in\mathbb C,\ i\in\mathbb N.
\end{equation}
By $PL\subset G_{f}(X)$ we denote the group generated by all such $g$. It is called the {\em group of piecewise linear paths in} $\mathbb C^{\infty}$.\footnote{This reflects the fact that the first integrals of the vectors of coefficients of formal equations \eqref{e1} corresponding to elements of $PL$ are piecewise linear paths in $\mathbb C^{\infty}$.} It was shown in [Br2, Proposition 3.14] that the group $\widehat{\cal C}_{PL}:=PL\cap \widehat{\cal C}_f$ of piecewise linear centers is dense in $\widehat{\cal C}_f$. It was also asked about the structure of the set of centers represented by piecewise linear paths in $\mathbb C^n$ (i.e., represented by products of elements of form \eqref{eq2.1} with all $c_i=0$ for $i>n$). In this paper we solve this problem for the group of rectangular paths.\\
\\
{\bf 2.2.} Let $X_{rect}\subset X$ be a semigroup generated by elements $a\in X$ whose first integrals $\widetilde a$ are rectangular paths in $\mathbb C^{\infty}$ consisting of finitely many segments parallel to the coordinate axes. The image of $X_{rect}$ under homomorphism $E$ is called {\em the group of rectangular paths} and is denoted by $G(X_{rect})$. Clearly, $G(X_{rect})\subset PL$. It is generated by elements  $e^{(a_nT)\, X_n\,t^n}\in G_f(X)$, $a_n\in\mathbb C$, $n\in\mathbb N$. 
Since there are no nontrivial relations between these elements (with $a_n\neq 0$), the group $G(X_{rect})$ is isomorphic to the free product of countably many copies of $\mathbb C$. Also, the group $G(X_{rect})\, (\subset G(X))$ is dense in $G_f(X)$ and the elements $X_n\, t^n$, $n\in\mathbb N$, form a generating subset of the Lie algebra ${\cal L}_{Lie}$, see [Br2]. Moreover, for each $g\in G(X_{rect})$ the first return map $\widehat P(g)\in G_c[[r]]$ can be explicitly computed and represents an algebraic function. 
Specifically, for the equation
\begin{equation}\label{e2.1}
\frac{dv}{dx}=a_n v^{n+1}
\end{equation}
corresponding to the element $g_n:=e^{(a_nT)\, X_n\,t^n}$ an explicit calculation shows that its first return map is given by the formula
\begin{equation}\label{e2.2}
\begin{array}{l}
\displaystyle
\widehat P(g_n)(r):=\frac{r}{\sqrt[n]{1-na_n T r^n}}\\
\\
\displaystyle
=r+\sum_{j=1}^{\infty}\frac{(-1)^{j}(n(j-1)+1)(n(j-2)+1)\cdots 1}{j!}(a_n T)^j r^{nj+1}.
\end{array}
\end{equation}
Here $\sqrt[n]{\cdot}:\mathbb C\setminus (-\infty,0]\to\Co$ stands for the principal branch of the power function. Then for a generic $g\in G(X_{rect})$, the first return map
$\widehat P(g)$ is the composition of series of form (\ref{e2.2}).

Our main result is
\begin{Th}\label{te1}
The restriction $\widehat P|_{G(X_{rect})}: G(X_{rect})\to G_c[[r]]$ is a monomorphism. In particular, $\widehat {\cal C}\cap G(X_{rect})=\{1\}$ and ${\cal C}\cap X_{rect}\subset {\cal U}$. Moreover, ${\cal C}\cap X_{rect}$ consists of elements $a\in X_{rect}$ such that their first integrals $\widetilde a:[0,T]\to\mathbb C^{\infty}$ are rectangular paths modulo cancellations\footnote{i.e., forgetting sub-paths of a given path consisting of a segment and then immediately of the same segment going in the opposite direction.} representing the constant path $[0,T]\to (0,0,\dots)\in\mathbb C^{\infty}$.
\end{Th}
\begin{proof}
The proof of Theorem \ref{te1} is based on the deep result of S.\,Cohen [C].

Consider an irreducible word $g=g_{k_1}\cdots g_{k_l}\in G(X_{rect})$ where $g_{k_i}:=e^{(a_{k_i}T)\, X_{k_i}\,t^{k_i}}$
and $a_{k_i}\in\mathbb C^{*}\, (:=\mathbb C\setminus\{0\})$. We must show that $\widehat P(g)=\widehat P(g_{k_1})\circ\dots\circ \widehat P(g_{k_l})\neq 1$; here $1$ is the unit of the group $G_c[[r]]$. Assume, on the contrary, that $\widehat P(g)=1$. Then from equation \eqref{e2.2} for all $r\in\mathbb C$ sufficiently close to $0$ we obtain
\begin{equation}\label{e2.3}
\begin{array}{l}
\displaystyle
\widehat P(g)(r)\\
\\
\displaystyle
=\frac{r}{\left(\left(\dots\left(\left(1+ b_{k_l}r^{k_l}\right)^{\frac{k_{l-1}}{k_{l}}}+ b_{k_{l-1}}r^{k_{l-1}}\right)^{\frac{k_{l-2}}{k_{l-1}}}+\cdots \right)^{\frac{k_{1}}{k_{2}}}+b_{k_1}r^{k_1}\right)^{\frac{1}{k_1}}}\,.
\end{array}
\end{equation}
Here we set $b_{k_i}:=-k_i a_{k_i}T$.

Making the substitution $t=\frac{1}{r}$ from equations $\widehat P(g)=1$ and \eqref{e2.3} we get for all sufficiently large positive $t$
\begin{equation}\label{e2.4}
\left(\left(\dots\left(\left(t^{k_l}+ b_{k_l}\right)^{\frac{k_{l-1}}{k_{l}}}+ b_{k_{l-1}}\right)^{\frac{k_{l-2}}{k_{l-1}}}+\cdots \right)^{\frac{k_{1}}{k_{2}}}+b_{k_1}\right)^{\frac{1}{k_1}}=t
\end{equation}
Here from the irreducibility of $g$ we obtain that $k_{i}\neq k_{i+1}$ for all $1\leq i\le l-1$.

Consider the multi-valued algebraic function over $\mathbb C$ defined by the left-hand side of equation \eqref{e2.4}. Then there exist a connected Riemann surface $S$, a finite surjective holomorphic map $\pi:S\to\mathbb C$ and a (single-valued) holomorphic function $f$ defined on $S$ such that the pullback by $\pi^{-1}$ of the restriction of $f$ to a suitable open subset of $S$ corresponds to the branch of the original function satisfying \eqref{e2.4} defined on an open subset  of $\mathbb C$ containing a ray $[R,\infty)$ for $R$ sufficiently large. Equation \eqref{e2.4} implies that $f$ coincides with the pullback $\pi^*z$ of the function $z$ on $\mathbb C$.

Let $P$ be the abelian group of maps $\mathbb C\to\mathbb C$ generated by $\{x\mapsto x^p\, ;\,p\in\mathbb N\}$. Here the inverse of the complex map $x\mapsto x^p$ is the map $x\mapsto x^{1/p}$ which maps a nonzero $x$ onto the unique complex number $y$ with $0\le\arg y<\frac{2\pi}{p}$ such that $y^p=x$. Let $T_{\mathbb C}$ be the abelian group of the maps $\{x\mapsto x+a\, ;\, a\in\mathbb C\}$. Then Theorem 1.5 of [C] states that the group of complex maps of $\mathbb C$ generated by $P$ and $T_{\mathbb C}$ is their free product $P*T_{\mathbb C}$.

Next, the function $h:\mathbb C\to\mathbb C$ defined by the left-hand side of \eqref{e2.4} belongs to the group generated by $P$ and $T_{\mathbb C}$ (in the definition of $h$ we define the fractional powers as in the above cited theorem). In turn, by the definition of $S$ there exists a subset $U$ of $S$ such that $\pi:U\to\mathbb C$ is a bijection and
$f\circ(\pi|_{U})^{-1}=h$. Since $f=\pi^*z$, the latter implies that $h(t)=t$ for all $t\in\mathbb C$. But according to our assumptions the word in $P*T_{\mathbb C}$ representing $h$ is irreducible. Thus it cannot be equal to the unit element of this group.

This contradiction shows that $\widehat P(g)\neq 1$ and proves the first statement of the theorem. 

The other statements follow straightforwardly from the corresponding definitions. We leave the details to the reader.
\end{proof}
\begin{R}\label{re1}
{\rm Theorem \ref{te1} implies that $G[[r]]$ contains a dense subgroup isomorphic to the free product of countably many copies of $\mathbb C$ generated by series of form \eqref{e2.2}. In fact, the subgroup of $G[[r]]$ generated by series of form \eqref{e2.2} with $n=1$ and $n=2$ only is already dense in $G[[r]]$. It follows, e.g., from [Br2, Proposition 3.11].
}
\end{R}

\noindent {\bf 2.3.} Let $a\in X_{rect}$ be such that  $\widehat E(\pi(a))=e^{a_{k_{1}}X_{k_1}t^{k_1}}\cdots e^{a_{k_l}X_{k_l}t^{k_{l}}}\in G_{f}(X)$ for some $a_{k_{1}},\dots, a_{k_l}\in\mathbb C$, i.e., the path $\widetilde a: [0,T]\to\mathbb C^{\infty}$ consists of $l$ segments parallel to the coordinate axes $z_{k_1},\dots, z_{k_l}$ of $\mathbb C^{\infty}$.
Considering $a_{k_1},\dots, a_{k_l}$ as complex variables in $\mathbb C^l$ we obtain a family ${\cal F}$ of rectangular paths. The first return maps $\widehat P(a)$ of elements of ${\cal F}$ can be computed by expanding the functions $e^{a_{k_{1}}X_{k_1}t^{k_1}}\cdots e^{a_{k_l}X_{k_l}t^{k_{l}}}$ in infinite series in variables $X_{k_s}t^{k_s}$, $1\le s\le l$, then replacing each $X_{k_s}$ by $DL^{k_{s}-1}$ where $D$ and $L$ are the differentiation and the left translation in the algebra of formal power series $\mathbb C[[z]]$, and then evaluating the resulting series in $D, L, t$ at elements $z^p$, see [Br1] for similar arguments. As a result we obtain (with $t$ substituted for $r$)
\begin{equation}\label{e2.5}
\begin{array}{l}
\displaystyle
\widehat P(a)(r)=r+\sum_{i=1}^{\infty}\left(\,\sum_{k_{1}s_1+\cdots +k_{l}s_l=i} q_{k_1;s_1,\dots, k_l;s_l}(i)\,\frac{a_{k_1}^{s_1}}{s_1!}\cdots \frac{a_{k_{l}}^{s_l}}{s_l!}\right)r^{i+1}\\
\\
\displaystyle
q_{k_{1};s_1,\dots, k_{l};s_l}(i):=\prod_{n=0}^{l-1}
\left\{\prod_{j=0}^{s_{n+1}}\left(i-\sum_{m=0}^{n}s_mk_{m} -jk_{n+1}+1\right)\right\}
\end{array}
\end{equation}
(here we set for convenience $s_0k_0:=0$).

By $c_i(a_{k_1},\dots, a_{k_l})$ we denote the coefficient at $r^{i+1}$ of $\widehat P(a)$. It is a holomorphic polynomial on $\mathbb C^l$.
The center set $C$ of equations \eqref{e1} corresponding to the family ${\cal F}$ is the intersection of sets of zeros $\{(a_{k_1},\dots, a_{k_l})\in\mathbb C^l\, ;\, c_{i}(a_{k_1},\dots, a_{k_l})=0\}$ of all polynomials $c_i$. According to Theorem \ref{te1} $(a_{k_1},\dots, a_{k_l})\in C$ if and only if the word $e^{a_{k_{1}}t^{k_1}X_{k_1}}\cdots e^{a_{k_l}t^{k_{l}}X_{k_l}}=I$ in ${\cal A}$. Since the groups generated by elements $e^{a_{k_{p_1}}t^{k_{p_1}}X_{k_{p_1}}},\dots , e^{a_{k_{p_m}}t^{k_{p_m}}X_{k_{p_m}}}$ with mutually distinct $X_{k_{p_j}}$ and nonzero numbers $a_{k_{p_j}}$ are free, the last equation implies that $C$ is the union of finitely many complex subspaces of $\mathbb C^{l}$. (For instance, if all $k_{j}$ are mutually distinct, then $C=\{0\}\subset\mathbb C^l$.)

Our next result gives an effective bound on the number of coefficients in \eqref{e2.5} determining the center set $C$.
\begin{Th}\label{te2}
The set $C\subset\mathbb C^{l}$ is determined by equations
$c_1=0,\dots, c_{d+1}=0$ where
$$
d:=\prod_{i=1}^{l-1}\frac{k_{i}}{gcd(k_i,k_{i+1})}
$$
(here $gcd(n,m)$ is the greatest common divisor of natural numbers $n$ and $m$).
\end{Th}
\begin{proof}
Since $c_{i}(z^{k_1}a_{k_1},\dots, z^{k_l}a_{k_l})=z^{i}c_i(a_{k_1},\dots, a_{k_l})$, $i\in\mathbb N$,
it suffices to prove that $C\cap B$ where $B\subset\mathbb C^l$ is the open unit Euclidean ball is determined by equations $c_1=0,\dots, c_{d+1}=0$. 

Next, there exists a positive number $R$ such that for each $(a_{k_1},\dots, a_{k_l})\in B$ the first return map $\widehat P(a)$ given by \eqref{e2.5} determines a holomorphic function on $\Di_R:=\{z\in\mathbb C\,;\, |z|<R\}$. On the other hand, according to \eqref{e2.2}, $\widehat P(a)$ is the composite of algebraic functions. Now, from formula \eqref{e2.3} we obtain straightforwardly that equation
$f(a)(r)=c$, $f(a)(r):=\frac{\widehat P(a)(r)-r}{r}$, has at most $d$ complex roots in $\Di_R$ (counted with their multiplicities), i.e., the valency of $f(a)$ on $\Di_R$ is at most $d$. From here and the result of Haymann [H, Theorem 2.3] we obtain
\begin{equation}\label{eq2.6}
\sup_{r\in\Di_{\frac{R}{2}}}|f(a)(r)|\le  e^{\frac{1+\pi^2 d}{2}}\cdot\frac{R^{d+1}-1}{R-1}\cdot\max_{1\leq k\leq  d+1}|c_k(a_{k_1},\dots, a_{k_l})|.
\end{equation}

This implies that if $c_1(a_{k_1},\dots, a_{k_l})=\dots =c_{d+1}(a_{k_1},\dots, a_{k_l})=0$, then $f(a)\equiv 0$, i.e., 
$(a_{k_1},\dots, a_{k_l})\in C$.

The proof is complete.
\end{proof}
\begin{R}\label{r2}
{\rm It follows from equations $c_{i}(z^{k_1}a_{k_1},\dots, z^{k_l}a_{k_l})=z^{i}c_i(a_{k_1},\dots, a_{k_l})$, $i\in\mathbb N$, inequality \eqref{eq2.6} and the Cauchy integral formula for derivatives of $f(a)$ that for any $k\in\mathbb N$ and all $\lambda\in\mathbb C^l$
$$
|c_{d+1+k}(\lambda)|\leq e^{\frac{1+\pi^2 d}{2}}\cdot\frac{R^{d+1}-1}{R-1}\cdot\left(\frac{2\sqrt{l}}{R}\right)^{d+k}\cdot (1+\|\lambda\|_2)^{d+k}
\cdot\max_{1\leq k\leq  d+1}|c_k(\lambda)|
$$
where $\|\cdot\|_2$ is the Euclidean norm on $\mathbb C^l$.\\
This and Proposition 1.1 of [HRT] imply that each polynomial $c_{d+1+k}$ belongs to the integral closure $\overline{I}$ of the polynomial ideal $I$ generated by $c_1,\dots, c_{d+1}$. 

An interesting question is: in which cases $\overline{I}=I$? (If this is true, then the (Bautin) polynomial ideal generated by all polynomials $c_i$ is, in fact, generated by
$c_1,\dots, c_{d+1}$ only.)
}
\end{R}

\end{document}